\documentclass[11pt]{article}
\usepackage[a4paper]{geometry}
\usepackage{amssymb,amsmath,amsthm}
\usepackage{cleveref}
\usepackage{enumerate}
\usepackage{pgf,tikz}
\usetikzlibrary{patterns}
\usepackage{float}

\newcommand{\abs}[1]{\left\lvert{#1}\right\rvert}
\newcommand\DOI[1]{{\tt DOI:#1}}
\newcommand\arXiv[1]{{\tt arXiv:#1}}
\newcommand\abbr[1]{\textsc{#1}}
\newcommand\junk[1]{}
\def\K{\mathbb{K}}

\def\Ga{\Gamma^+}

\newtheorem{theorem}{Theorem}
\newtheorem*{theorem*}{Theorem}

\newtheorem{lemma}[theorem]{Lemma}
\newtheorem{corollary}[theorem]{Corollary}

\newtheorem{conjecture}[theorem]{Conjecture}

\theoremstyle{definition}
\newtheorem{example}[theorem]{Example}

\title{On the Small Quasi-kernel conjecture --- a survey\footnote{A preliminary version of this survey was published as \cite{smallQK}.}}
\usepackage[affil-sl]{authblk}
\author[,1]{P\'eter L. Erdős{\protect\footnote{The authors were supported in part by the National Research, Development and Innovation Office -- NKFIH grant SNN~135643, K~132696.}}}
\author[$\dag$,1]{Ervin Győri}
\author[$\dag$,1]{Tam\'as Róbert Mezei}
\affil{Alfréd Rényi Institute of Mathematics (HUN-REN), Hungary\\
\texttt{$\langle$erdos.peter,gyori.ervin,mezei.tamas.robert$\rangle$@renyi.hun-ren.hu}}
\author[$\dag$,2]{Nika Salia}
\affil{King Fahd University of Petroleum $\&$ Minerals, Dhahran, Saudi Arabia\\ \texttt{salianika@gmail.com}}
\author{Mykhaylo~Tyomkyn{\protect\footnote{Supported in part by ERC Synergy Grant DYNASNET 810115 and GA\v{C}R Grant 22-19073S.}}}
\affil{Department of Applied Mathematics, Charles University\\  \texttt{tyomkyn@kam.mff.cuni.cz}}

\begin{document}
\maketitle

\begin{abstract}
An independent vertex subset $S$ of the directed graph $G$ is a \emph{kernel} if the set of out-neighbors of $S$ is $V(G)\setminus S$. An independent vertex subset $Q$ of $G$ is a \emph{quasi-kernel} if the union of the first and second out-neighbors contains $V(G)\setminus S$ as a subset. Deciding whether a directed graph has a kernel is an NP-hard problem. In stark contrast, each directed graph has quasi-kernel(s) and one can be found in linear time. In this article, we will survey the results on quasi-kernel and their connection with kernels. We will focus on the \emph{small quasi-kernel} conjecture which states that if the graph has no vertex of zero in-degree, then there exists a quasi-kernel of size not larger than half of the order of the graph. The paper also contains new proofs and some new results as well.
\end{abstract}

\section{Introduction}\label{sec:intro}

\noindent Let $G=(V(G), E(G))$ be a finite directed graph (\emph{digraph} for short) on $n$ vertices, without loops and multi-edges going in the same direction; anti-parallel edges are allowed. An edge $uv\in E(G)$ is oriented from its \emph{tail} $u$ towards its \emph{head} $v$, hence $v$ is an \emph{out-neighbor} of $u$, and $u$ is an \emph{in-neighbor} of $v$. Let $\Ga(v)$ denote the set of out-neighbors of a vertex $v\in V(G)$, and let $\Ga(V'):= \bigcup \{  \Ga(v) : v\in V'\}$ for any $V' \subseteq V$.
We use the notation $\Ga_1(V'):= V' \cup \Ga(V')$ and $\Ga_2(V'):=  V' \cup \Ga(V') \cup \Ga(\Ga(V'))$. The analogous notions for in-neighbors are denoted by $\Gamma^-(v), \Gamma^-(V'), \Gamma_1^-(V')$ and $\Gamma_2^-(V')$.

An independent subset of vertices $K \subseteq V(G)$ is called a \emph{kernel}  if $ \Ga_1(K)=V(G)$. Throughout the history of this subject, various objects have been defined and named in different ways, leading to a lack of complete consensus on the matter. While, the notion itself was introduced by von Neumann and Morgenstern~\cite{game}, as an independent subset $K\subseteq V(G)$ satisfying $\Gamma^-_1(K) = V(G)$, in this work we follow the definition which uses the inverse orientation, used e.g. by Bondy (\cite{Bondy} or Chv\'atal and Lov\'asz(\cite{CL74}).

An independent vertex set  $Q \subseteq V(G)$ is called a \emph{quasi-kernel}, or \abbr{Qk} for short, if $\Ga_2(Q)=V(G)$. Note that, Chv\'atal and Lov\'asz~\cite{CL74} use the term \emph{semi-kernel} for quasi-kernel.
As Landau showed already in 1955~\cite{landau}, every tournament has a quasi-kernel. Indeed, it is easy to see that any vertex of maximum out-degree has this property, see~\cite[Exercise {\S 8.6}]{lovasz}.
In 1962 Moon~\cite{Moon} proved, that if there is no \emph{source} in a tournament $G$, that is, each vertex has an in-neighbor, then the tournament contains at least three different \abbr{Qk}'s. Note that, Moon refers to such vertices as a \emph{king} of the tournament. Although there is a significant body of literature studying kings in tournaments, we will not provide an overview of it in this work.

A cyclically oriented odd cycle  demonstrates that  digraphs do not always have kernels. Even more, finding a kernel in a  digraph is proved to be an \textsc{NP}-hard problem by Chv\'atal \cite{CL73} (see also \cite{GJ}).
Given these facts, it may come as a surprise that this is not the case for quasi-kernels.
\begin{theorem}[Chv\'atal and Lov\'asz, 1974,~\cite{CL74}]\label{th:CL}
    Every digraph contains a quasi-kernel.
\end{theorem}
\noindent Shortly afterwards P.L.~Erd\H{o}s and L.A.~Sz\'ekely proposed the following conjecture:
\begin{conjecture}[Small Quasi-kernel conjecture~\cite{Fete}, 1976] \label{conj:SQK}
Every digraph $G$ with no sources contains a quasi-kernel with  at most $|V(G)|/2$ vertices.
\end{conjecture}
\noindent We remark that the conjecture can not be strengthened, as a cyclically oriented $C_4$ shows. Furthermore, Example \ref{th:tight} will provide large, strongly connected, source-less digraphs where every \abbr{Qk} contains roughly $|V(G)|/2$ vertices. We shall use the term \emph{small quasi-kernel} referring to a \abbr{Qk} of order at most $|V(G)|/2$. So the conjecture states that every digraph with no sources  contains a small \abbr{Qk}.

The conjecture was initially not published in any paper and was communicated by word of mouth until it finally appeared in 2010~\cite{Fete}. After a long hibernation state, the conjecture was revitalised by Kostochka, Luo, and Shan in 2020 (see \cite{Kost}). In the last three years, a lot of effort was spent on proving the small \abbr{Qk} conjecture. However, the conjecture is still wide open. For example, it has not even been proved that there exists some constant $c<1$ for which every digraph $G$ with no sources contains a \abbr{Qk} with at most $c|V(G)|$ vertices.

In this paper we we will survey our knowledge on quasi-kernels. It also will provide some new results and proofs.

\section{The existence of quasi-kernel(s)}\label{sec:exist}
Let us start with a proof of the Chv\'atal-Lov\'asz theorem. The original induction proof can be paraphrased as a fast two-phase greedy algorithm, which we shall call the \emph{Chv\'atal-Lov\'asz-algorithm} (or \emph{CL-algorithm} for short). We present this version of the proof because we will rely on it later on.

\begin{proof}[Proof of Theorem \ref{th:CL}]
Let $v_1,\ldots,v_n$ be an arbitrary order on the vertices of $G$. In the first phase, let $i_1=1$ and for $k\geq 2$ let $i_k$ be the minimum index in
\begin{equation}\label{eq:CLremaining}
V(G)\setminus \Gamma_1^+(\{v_{i_1},\ldots,v_{i_{k-1}}\});
\end{equation}
repeat this until the above set~\eqref{eq:CLremaining}  becomes empty. Let $t$ be the maximum for which $i_k$ is defined, and we call ${(v_{i_k})}_{k=1}^t$ the vertices selected in the first phase. For the second phase, let $G':= G[v_{i_1},\ldots,v_{i_{t}}]$ and repeat the previous selection procedure in $G'$ with the following order of the vertices  $v_{i_t},v_{i_{t-1}}, \dots,v_{i_1}$. It is easy to see that the selected vertex subset of $G'$ forms a quasi-kernel in $G$.
\end{proof}
\noindent The proof clearly shows that if the digraph contains a vertex of in-degree $0$ (that is, a source), then this vertex must belong to every \abbr{Qk}. It is interesting to point out that the presented proof is essentially the only one we know. The proof of Bondy~\cite{Bondy} is a version of the same argument. In his paper, Bondy also described a successful  argument of S. Thomass\'e. Croitoru (in \cite{CC15}) gave a detailed description of that proof. The latter author also pointed out, that the Thomass\'e algorithm cannot always construct every \abbr{Qk}. As we will show in Example \ref{th:nonCL}, the Chv\'atal-Lov\'asz algorithm also has the same ``shortcoming''.

While finding a \abbr{Qk} in a digraph is computationally not hard,  deciding whether there exists a \abbr{Qk} containing a fixed vertex is \textsc{NP}-hard. It was proved by C. Croitoru using a polynomial time reduction from the CNF satisfiability problem \textbf{SAT} (see~\cite{CC15}).

How many (different) \abbr{Qk} can be in a digraph? This question was raised at first by Jacob and Meyniel (\cite{JM96}).
\begin{theorem}[Jacob and Meyniel, 1996]
Every digraph without a kernel has at least three \abbr{Qk}s. For each  integer $n\ge 3$, there exists a graph without a kernel and with exactly three \abbr{Qk}s.
\end{theorem}
\noindent Gutin and his colleagues characterized digraphs with exactly one \abbr{Qk}:
\begin{theorem}[Gutin, Koh, Tay, and Yeo, 2004, \cite{gutin}]
A digraph has exactly one \abbr{Qk} if and only if the set of sources constitutes a kernel.
\end{theorem}
\noindent Coitoru provided a different proof in \cite{CC15} for this result, based on the Thomass\'e algorithm. Furthermore, it was also proved that there are digraphs, where this algorithm necessarily misses some \abbr{Qk}s.

Gutin, Koh, Tay, and Yeo also characterized source-less digraphs with exactly two \abbr{Qk}s. Coitoru extended this question to all digraphs, and came up with the following result:
\begin{theorem}[Coitoru, \cite{CC15}]
If a digraph has exactly two \abbr{Qk}s, then at least one of them is a kernel and their intersection consists of all sources.
\end{theorem}
\noindent Gutin and his colleagues conjectured in \cite{Gconj} that if a digraph is source-less, then there always exist at least two disjoint \abbr{Qk}s. However, in \cite{gutin} they refuted this statement, providing a counterexample on 14 vertices.

Langlois, Meunier, Rizzi, and Vialette recognized that the situation was even worse, and they proved the following result:
\begin{theorem}[Langlois, Meunier, Rizzi and Vialette \cite{LMRV21}, 2021]\label{th:Langlois}\hfill \
\begin{enumerate}[{\rm (i)}]
\item It is computationally hard to decide whether a digraph contains two disjoint \abbr{Qk}s.
\item Computing the smallest \abbr{Qk} in a digraph is also computationally hard.
\end{enumerate}
\end{theorem}
\noindent To finish this section we remark that Heard and Huang provided several different extra conditions for source-less digraphs to ensure that they contain at least two vertex disjoint \abbr{Qk}s (see \cite{Heard08}). These are the semicomplete multipartite digraphs (including tournaments), the quasi-transitive digraphs, and finally the locally semicomplete digraphs.

\section{The Small Quasi-kernel conjecture}\label{sec:quasikernel}
\subsection{The Chv\'atal-Lov\'asz algorithm}

One might assume that, given a directed graph without sources, there exists a vertex ordering that leads to a small \abbr{Qk} produced by the CL-algorithm. However, as shown by the following example, this assumption is incorrect. Of course, this also means that not all \abbr{Qk} can be found by the CL-algorithm. In contrast, Lemma~\ref{th:observ} demonstrates a specific case in which a slightly modified CL-algorithm does yield a small \abbr{Qk}.

\begin{example} \label{th:nonCL}
    Define our directed graph as follows:
    \begin{gather*}
        V(G)=\{v,u,w,v_i,u_i,w_i:i\in[k]\},\\
        E(G)=\{{vu},{uv},{vw},{wv},{uw},{wu}\}\cup \{{uu_i},{u_i u},{vv_i},{v_i v},{ww_i},{w_i w}:i\in[k]\}.
    \end{gather*}
\end{example}

\paragraph{Analysis of the example.} It is easy to see that any vertex from the set $\{u,v,w\}$ alone is a quasi-kernel. However, there exists no ``right'' order for the CL-algorithm that would  produce a quasi-kernel of order at most $\frac{3k+3}{2}$.

Indeed, if the first pick of the CL-algorithm is  a vertex from $\{v,u,w\}$, say $v$, then the algorithm will collect all vertices  $\{u_i,w_i:i\in[k]\}$ in some order and terminate. So the algorithm will output the quasi-kernel $\{v\}\cup \{u_i,w_i:i\in[k]\}$ of size $2k+1$.

If the first vertex chosen by the CL-algorithm is not a vertex from  $\{v,u,w\}$, say $v_k$, then after it has been picked we obtain $k-1$ isolated vertices $\{v_i:i\in[k-1]\}$, as their only neighbor has just been removed. Thus, all vertices $\{v_i:i\in[k]\}$ will be included in the quasi-kernel.
Now we consider the rest of the graph: if the algorithm first picks a vertex $u$ or $w$, say $u$, then the quasi-kernel will have to contain all vertices from $\{w_i:i\in[k]\}$, resulting in a \abbr{Qk} of size $2k$. Otherwise, if the algorithm first picks a leaf vertex, say $u_k$, then we will need to select all of $\{u_i:i\in[k]\}$ for the \abbr{Qk}, again resulting in a large quasi-kernel of size at least $2k$. \qed{}

\medskip

\noindent That said, in some situations, a version of the CL-algorithm does result in a small \abbr{Qk}.
\begin{lemma} \label{th:observ}
Assume that digraph $G$ does not contain a source and it satisfies the following property: If $i < j$ and $v_jv_i \in E$ then $v_iv_j\in E$ holds as well.  Then there is a small \abbr{Qk} in $G$.
\end{lemma}
\begin{proof}
We only run the first phase of the CL-algorithm and modify it as follows. If the vertex set in~\eqref{eq:CLremaining} induces no  edges in $G$, we terminate. Otherwise, we select $v_{i_k}$ from the set~\eqref{eq:CLremaining}  with the smallest subscript such that $v_{i_k}$ has an out-neighbor in~\eqref{eq:CLremaining}.

Suppose the modified selection phase terminates after $v_{i_t}$, we claim that $Q:={\{v_{i_k}\}}_{k=1}^t$ is a small quasi-kernel. Indeed, as $k$ increases by $1$, the cardinality of~\eqref{eq:CLremaining} decreases by at least 2, so $t\le |V(G)|/2$, thus $Q$  is small. Next, we show that $Q$ is an independent set. By definition for $i_j < i_k$ we have no $v_{i_j}v_{i_k} \in E$ otherwise at the time of the selection of the $i_k$th vertex the vertex $v_{i_k}$ would not be in~\eqref{eq:CLremaining}. And if $v_{i_k}v_{i_j}\in E$ then, by definition, we also have $v_{i_j}v_{i_k} \in E$, a contradiction. Finally, for $k=t$, the vertices remaining in \eqref{eq:CLremaining} must be second neighbors of $Q$, since they have lost their in-neighbors. So $\Ga_2(Q) = V(G)$ and it is a small \abbr{Qk} indeed.
\end{proof}

\subsection{The Kostochka - Luo - Shan conjecture}\label{sec:Kost}
\noindent In 2020 Kostochka, Luo, and Shan~\cite{Kost} reinvigorated the study of the small \abbr{Qk} conjecture. They introduced a generalization of the small \abbr{Qk} conjecture to digraphs with possible sources:
\begin{conjecture}[Kostochka, Luo, and Shan~\cite{Kost}]
Let $G$ be a digraph, and let $S\subseteq V({G})$ be the set of the vertices with zero in-degree in $G$. Then there exists a quasi-kernel $Q$ with cardinality
\begin{equation}\label{eq:kost}
        |Q|\le \frac{|V(G)| + |S| - |\Ga(S)|}{2}.
\end{equation}
\end{conjecture}
\noindent We say that a digraph is \emph{kernel-perfect} if every induced sub-digraph of it has a kernel. As Richardson showed in 1953 (\cite{perfect}), if there is no directed odd cycle in a digraph then it is kernel-perfect. This property is useful in our discussion:
\begin{theorem}[Kostochka, Luo and Shan~\cite{Kost}]\label{th:Kost}
Suppose that in the digraph $G$, the vertex subset $V(G)\setminus \Ga_1(S)$ has a partition $V_1 \cup V_2$ such that the induced sub-digraphs $V_1[G]$ and $V_2[G]$ are kernel-perfect. Then  $G$  satisfies the KLS conjecture.
\end{theorem}
\noindent Due to Richardson's theorem, Kostochka, Luo and Shan also proved the following consequence of the previous theorem:
\begin{corollary}
The KLS conjecture holds for every orientation of each graph with a chromatic number of 4. By the Four Color  Theorem, this implies that every orientation of each planar graph satisfies the conjecture.
\end{corollary}
Finally, they proved that if the KLS conjecture fails, then a minimum size counterexample does not contain a source. That leads to the following, very nice result:
\begin{theorem}[Kostochka, Luo, and Shan, \cite{Kost}]\label{th:KLS}
The KLS conjecture and the small \abbr{Qk} conjecture are in fact equivalent.
\end{theorem}

\bigskip\noindent In 2023 Ai, Gerke, Gutin, Yeo and Zhou (\cite{gerke}) developed a method to discover new digraph classes with the small \abbr{Qk} property, and also provide new, simpler proofs for a series of older results. For example, they reproved Theorem \ref{th:Kost}.

Let us define the class $\K$ of directed graphs in the following way: if $G\in \K$ then there is no induced directed subgraph in $G$ which is isomorphic to $\vec K_{1,4}$ or to $\vec K_{1,4}+e$. The first digraph is a four-star, the edges directed from the center to the leaves, the second one is the same plus one directed edge between two leaves.\footnote{The paper \cite{gerke} uses inverse notations comparing with ours. Therefore using their naming convention may cause misunderstanding, so we will ignore the names they used.} The paper's main result is the following:
\begin{theorem}[Ai, Gerke, Gutin, Yeo, and Zhou \cite{gerke}]\label{th:gerke}
Each source-free digraph in class $\K$ has a small \abbr{Qk}.
\end{theorem}
\noindent Another result deals with \emph{good} \abbr{Qk}. A quasi-kernels $Q$ has this property, if $Q\subset \Ga(\Ga(Q))$.
\begin{theorem}\label{th:good}
A source-free digraph with a good \abbr{Qk} satisfies the small \abbr{Qk} conjecture.
\end{theorem}
The proof one can find in the paper \cite{gerke} (similarly to the other proofs there) uses a minimal size \abbr{Qk} of the digraph. However, in light of Theorem \ref{th:Langlois}, finding such smallest example is not easy. Therefore here we insert our proof of this result. This starts with a good \abbr{Qk} and algorithmically finds the required small \abbr{Qk}.\footnote{Here we mention that while that preprint was published only in 2023, the mentioned results from that article were developed and proved well before the preprint version of \cite{gerke} was published. The main reason for the late publication was that we hoped to prove completely the small \abbr{Qk}, unfortunately without success.}

\begin{proof}
Assume that $Q$ is a good \abbr{Qk}. By definition, each vertex in $Q$ has an in-neighbor in $\Ga(Q)$. Therefore if $Q' \subseteq Q$ satisfies $\Ga(Q')=\Ga(Q)$, then $Q'$ is also a quasi-kernel. To define a suitable $Q'$, select vertices from $Q$ one by one, such that in each turn the number of out-neighbors of the set of vertices chosen so far increases strictly. When no additional vertices can be selected, every element of $\Ga(Q)$ is an in-neighbor of one of the selected vertices. The number of selected vertices is at most $\min(|Q|,|\Ga(Q)|)$, so this is a small \abbr{Qk}.
\end{proof}
This result surpasses a nice observation of van Hulst (\cite{Hulst}) from 2021: if a source-less digraph has a kernel, then it satisfies the small \abbr{Qk} conjecture.

Another result from \cite{gerke} is the following:
\begin{theorem}[Ai, Gerke, Gutin, Yeo, and Zhou]\label{th:main}
If a digraph $ G$ with no sources has a quasi-kernel $Q$ such that the induced subgraph on $V(G)\setminus \Ga_1(Q)$ admits a kernel $K$ (for example if $V(G)\setminus \Ga_1(Q)$ is an independent set), then $G$ has a small quasi-kernel.
\end{theorem}
They used this as a lemma for an alternative proof of Theorem \ref{th:Kost}. The proof in \cite{gerke} uses minimality and is therefore non-constructive. So here again, we insert our proof.

\begin{proof}
For simplicity let us use the notations $A=Q, B=\Ga{Q}$ and $C=V(G)\setminus \Ga_1$. We define further subsets of the vertices in the graph $G$ and create a visual representation to demonstrate them, see Figure~\ref{fig1}.
    \begin{itemize}
        \item Let $D:=A\cap \Ga(K)$. That is, $D$ is the set of out-neighbors of $K$ in $A$.
        \item Let $J:=\Ga(D)$. Note that, $J\subseteq B$ since $\Ga(A)=B$.
        \item Let $F:=(\Ga(J)\cap A)\setminus D$.
        \item Let $H:=\Ga(F)\setminus J$.
        \item Let $B':=B\setminus(J\cup H)$.
    \end{itemize}

    \begin{figure}[ht]
        \centering
        \begin{tikzpicture}[scale=.7]
            \draw [rounded corners] (6,0) rectangle (10,1);
            \node at (11,0.5) {\Large C};
            \draw [rounded corners] (3,3) rectangle (10,4);
            \node at (11,3.5) {\Large B};\node at (8,3.5) {\Large B'};
            \draw [rounded corners] (0,6) rectangle (10,7);
            \node at (11,6.5) {\Large A};
            \draw [pattern=dots] (7.5,6.5) ellipse (50pt and 7pt);
            \node at (9.7,6.5) {\Large A'};
            \draw[rounded corners,pattern=dots,pattern color=black!20] (8,0) rectangle (9.5,1);
            \node at (8.6,0.5) {\Large K};
            \draw[rounded corners,pattern=grid,pattern color=black!20] (0,6) rectangle (2.5,7);
            \node at (1.5,6.5) {\Large D};
            \draw[->,very thick] (8.7,0) .. controls (5,-3) and (0,4.5) .. (1,6);
            \draw[rounded corners,pattern=north east lines,pattern color=black!20] (3,3) rectangle (4.5,4);
            \node at (3.5,3.5) {\Large J};
            \draw[->,very thick] (2,6) -- (3.7,4);
            \draw[rounded corners,pattern=north west lines,pattern color=black!20] (2.7,6) rectangle (4.5,7);
            \node at (3.5,7.5) {\Large F};\node at (3.2,6.5) {\Large F${}_1$};\node at (4,6.5) {\Large F${}_2$}; \draw (3.5,6) -- (3.5,7);
            \draw[->,very thick] (4,4) -- (3.5,6);
            \draw[rounded corners,pattern=crosshatch,pattern color=black!20] (4.7,3) rectangle (6,4);
            \node at (5.3,3.5) {\Large H};
            \draw[->,very thick] (4,6) -- (5.4,4);
        \end{tikzpicture}
        \caption{A sketch of $G$; the set $A$ is a \abbr{Qk} of $G$, the set $B$ is out-neighborhood of $A$, and $C$ is the remaining set of vertices that are in the second out-neighborhood of $A$. We assume that $C$ has a kernel $K$. The sets $A$ and $B$ are partitioned into smaller subsets, as described in the proof of Theorem~\ref{th:good}.}\label{fig1}
        \end{figure}
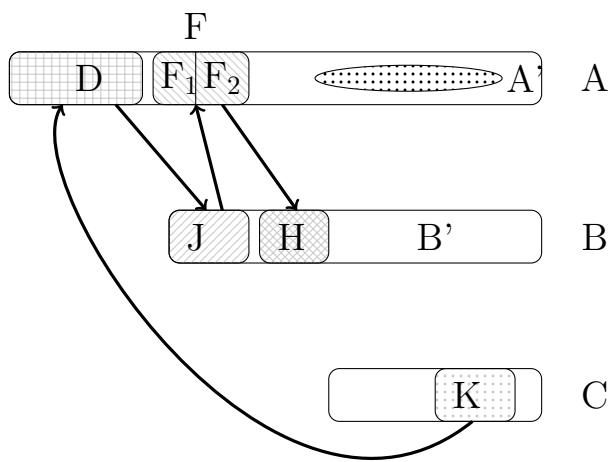

To find a small \abbr{Qk} in $G$, we define two quasi-kernels of $G$ and show that one of them is small.

The subset $D$ is a dominating set for $J$ and the subset $J$ is a dominating set for $ F$. The subset $F$ is a dominating set for $H$. Finally, $A\setminus(D\cup F)$ is a dominating set  for $B'$. Let $A'$ be a minimal subset of $ A\setminus (D \cup F)$ with the property, that $B' \subseteq \Ga(A')$, such subset exists since $B' \subseteq \Ga(A\setminus (D\cup F))$. Thus we have $\abs{A'} \le \abs{B'}$.

Next, let $F_2$ be a minimal subset of $F$ such that $H \subseteq \Ga(F_2)$. Again, this means that $\abs{F_2} \le \abs{H}$. We denote the remaining vertices of $F$ by $F_1=F \setminus F_2$.

Let $A'':=A \setminus (A' \cup F \cup G)$; we have $A''\subseteq \Ga(B'\cup H \cup (C\setminus K) )$, since all vertices in $A$ have a positive in-degree.

Now we define two different quasi-kernels. First, let $Q_1:= K \cup F \cup A'$. It is clear that $Q_1$ is an independent subset. Then $(C\setminus K) \cup D$ is dominated by $K$. Furthermore, $H$ is dominated by $F$. Finally, $B'$ is dominated by $A'$. Moreover, $J$ is two-dominated by~$K$. Finally, $A''$ is two-dominated by $K\cup F \cup A'$ via the subsets $C\setminus K$, $H$ and $B'$. So indeed, $Q_1$ is a \abbr{Qk}.

\noindent  The second \abbr{Qk} is $Q_2:= A \setminus F_1$. Indeed, $E = \Ga(D)$, and $H\subseteq \Ga(F_2)$, moreover $B' \subseteq \Ga(A\setminus [G\cup F])$. Furthermore $C\subseteq \Ga(B)$. So $Q_2$ is \abbr{Qk} indeed.

\medskip
\noindent In the rest of the proof, we will show that at least one of them is a small \abbr{Qk} of $G$.
We will prove this by contradiction. Assume that none of them is smaller than $n/2$; consequently, we have
\begin{align*}
        \abs{Q_1}&= \abs{K} + \abs{F} + \abs{A'} > \abs{D} + \abs{B} + (\abs{A} - \abs{F}  - \abs{A'} - \abs{D}) + \abs{C} -\abs{K}=\\
        &= \abs{A} - \abs{F} - \abs{A'} + \abs{B} + \abs{C} - \abs{K}.
\end{align*}
Thus we have
\begin{equation}\label{eq:egy}
    2(\abs{K} + \abs{A'} + \abs{F}) > \abs{A} + \abs{B} + \abs{C}.
\end{equation}
For $Q_2$ we have $\abs{Q_2}=\abs{A} - \abs{F_1} >  \abs{F_1} + \abs{B} + \abs{C}$, thus we have
\begin{equation}\label{eq:ket}
    \abs{A} > 2\abs{F_1} + \abs{B} + \abs{C}.
\end{equation}
Summing up \cref{eq:egy} and \cref{eq:ket} we get $\abs{K} + \abs{F_2} + \abs{A'} > \abs{B} + \abs{C}$.
However, as we saw earlier $\abs{F_2} \le \abs{H}$ and $\abs{A'} \le \abs{B'}$, therefore $\abs{F_2} + \abs{A'} \le \abs{B}$. Consequently, we get $\abs{K} > \abs{C}$, a contradiction.
\end{proof}

\bigskip\noindent We remark here, without going to detail, that in the paper \cite{gerke} the author also studied source-free, one-way split graphs. They did prove the conjecture for this graph class, showing that there always exists  \abbr{Qk} of size at most $\frac{n+3}{2}-\sqrt{n},$ furthermore they gave examples to show that this is sharp. In the paper \cite{LMRV23} Langlois, Meunier, Rizzi and Vialette studied unrestricted, source-free split graphs. For this case they proved, that there always exists a \abbr{Qk} of size at most $2n/3.$

\subsection{Sporadic results on small quasi-kernels}\label{sec:tour}
In this subsection, we survey results on digraphs $G$ similar to a tournament. We say that $G$ is a \emph{hairy tournament} if $V(G)$ is partitioned by two subsets $A$ and $I$ such that $G[A]$ is a tournament and each vertex in $I$ has exactly one incident edge: an incoming edge from $A$. The notion was introduced by  Hayes-Carver and  Witt~\cite{Hairy}.

\begin{theorem} [Hayes-Carver and  Witt~\cite{Hairy}]\label{th:hairy}
Let $G$ a hairy tournament with no sources. Then there exists a small \abbr{Qk}.
\end{theorem}
\noindent Although the poster~\cite{Hairy} hints at the possibility of using a minimal counterexample to create a smaller one and arrive at a contradiction, it does not provide formal proof. Therefore we present a distinct and succinct proof here.
\begin{proof}
Let $A=\{a_1,\ldots,a_m\}$ and let $b(a_i)= |\Ga(a_i) \cap I|$. So by the assumption, we have $|I|=\sum_i b(a_i)$.  We define an auxiliary  tournament $T$.

Let $T$ be a tournament on the same vertex set as $G$, containing $G$ as a subgraph. For simplicity, for each vertex $a_i \in A$  we denote the set of vertices $\{a_i\}\cup (\Gamma^+(a_i)\cap I)$ by $B_i$. For any $i$ and $j$, if $G$ contains the edge $a_i a_j$, then let $T$ contain all edges $\{b_i b_j :b_i\in B_i,b_j\in  B_j\}$. Finally, on each $B_i$, take an arbitrary $b(a_i)+1$ element transitive tournament such that the  vertex $a_i$ is the root (i.e., the unique source) of this transitive tournament.

Let $a$ be a maximum out-degree vertex in $T$, thus $|\Gamma^+(a)|\ge \frac{(|V(T)|-1)}{2}$. By the construction of $T$, the vertex $a$ is from  $A$, so $a=a_i$ for some $i$. As mentioned in Section~\ref{sec:intro}, the maximum out-degree vertex is a king in the blown-up tournament $T$. By definition, it means that $a_i$ is also a king within $G[A]$.

Let us define $Q\subseteq V(G)$ as follows:
\begin{equation*}
Q =\left\{a_i \right\}\cup \bigcup \left \{ \Ga(a_j) : a_i a_j \not\in E(G) \right \}.
\end{equation*}
$Q$ is a  \abbr{Qk} of $T$, since $a_i$ is a king of $G[A]$ and $Q$ is an independent set. Even more, $Q$ is  a  small \abbr{Qk} of $G$ since $|\Gamma_{T}^+(a_i)|\ge \frac{(|V(T)|-1)}{2}$.
\end{proof}

It is worth noting that Theorem~\ref{th:hairy} can be extended to a broader class of digraphs. Specifically, the same reasoning can be applied to the directed graph $G$ satisfying the following conditions: let $V(G) = A \biguplus I$ be a partition, and assume that $G[A]$ is a tournament, $\Gamma^+(I)= \emptyset$, and $\Gamma^-(I)\subseteq A$. In words: this is a hairy-tournament-like digraph, where each vertex in $I$ has a positive in-degree.

\vspace{2mm}
\noindent
Disjoint copies of oriented four-cycles and cycles of length two are the only known directed graphs with no sources  that have no QK of size less than half of its vertices. These examples demonstrate the sharpness of the small \abbr{Qk} conjecture. On the other hand, one might believe that all strongly connected directed graphs (which, by definition, do not have sources)  have asymptotically smaller \abbr{Qk}'s. We show, however, that this is not the case: we present below a large strongly connected digraph with no sources where all QKs have at least half of the  vertices asymptotically.

\begin{example}\label{th:tight}
Consider the following hairy tournament $G$: Let it be a regular tournament on $2n+1$ vertices, i.e., every vertex has in-degree $n$ and out-degree $n$. Furthermore, for each vertex $a_i$ in $G$ let  $b(v_i)=2n+1$. The constructed digraph $G$ has $(2n+2) \cdot (2n+1)$ vertices. Any quasi-kernel may contain at most one vertex from $G$. Therefore, the smallest quasi-kernel $Q$ contains $1+ n\cdot (2n+1)$ vertices. Then
\begin{equation*}
        \frac{|Q|}{|V(G)|}=\frac{1+ n\cdot (2n+1)}{(2n+2) \cdot (2n+1)}.
\end{equation*}
Thus for each positive constant $\varepsilon$, there is large enough $n$ such that $G$ does not contain a \abbr{Qk} of size smaller than $\left(\frac{1}{2}  - \varepsilon\right)\cdot\frac{V(G)}{2}$.

Now, if we add two vertices $w_1$ and $w_2$ such that $w_1$ is joined to every vertex of $I$ by an edge, an edge joins $w_1$ to $w_2$, and $w_2$ is joined by an edge to $a$ for exactly one $a\in A$. The obtained digraph $G'$ is strongly connected. It is clear, that any \abbr{Qk} in $G$ is also a \abbr{Qk} in $G'$. (The vertex $w_1$ is a first out-neighbor, and $w_2$ is a second out-neighbor of any vertex of $Q\cap I.)$ Let now $Q$ be an arbitrary \abbr{Qk} in $G'$. Then vertex $w_1$ cannot belong to $Q$. Assume the opposite. Then no vertex from $I$ can belong to $Q$ and there is no vertex in $A$ for which every vertex from $I$ is a second out-neighbor. Finally, $w_2$ can belong to $Q$ but it is still necessary that a king of $G[A]$ belongs to $Q$, and this defines which elements of $I$ should belong to $Q$. Therefore, in this case $Q$ is nothing else but the union of a \abbr{Qk} from $G$ and $w_2$. The graph $G'$ does not have a smaller quasi-kernel than $G$. \qed{}
\end{example}

\bigskip\noindent
In this subsection, we studied so far quite dense digraphs. In the next result, we consider a much sparser example.

A graph is called \emph{unicyclic} if it is connected and contains exactly one cycle. We call a digraph  \emph{unicyclic} if it is an orientation of a unicyclic graph.
In this work we resolve the conjecture for all $n$ vertex, directed unicyclic graphs with no sources, establishing a sharp bound $\frac{n+2}{3}$ on the minimum size of a  quasi-kernel.

\begin{theorem}\label{thm:uniciclic}
Let $G$ be an $ n$-vertex directed unicyclic graph without sources.
Then there exists a \abbr{Qk} of size at most $\frac{n+2}{3}$ sharp.
\end{theorem}
\begin{proof}
Note, that every orientation of a unicyclic graph with no sources contains a directed cycle $v_1v_2\dots v_{\ell}v_1$ for some integer $\ell$, and for every $i=1,\ldots,\ell$ there is a directed tree rooted at $v_i$ that shares exactly one vertex with the cycle, and all of its edges are directed away from $v_i$. Thus let $G$ be a directed cycle $v_1v_2\dots v_{\ell}v_1$ and vertex-disjoint rooted spanning trees $T_{v_i}$ (rooted at $v_i$) for $i\in [n]$.

Let $T_v$ be a directed tree that has a single source $v$. We recursively define a vertex coloring of $T_v$ using three colors. For each $i \in [3]$, we define a coloring function $c_{i}:V(T_v)\rightarrow [3]$ that assigns a color to every vertex of $T_v$, subject to the following conditions: $c_i(v)=i$ and for every vertex $u\neq v$ in $T_v$ we require $c_i(u)\equiv c_i(u^-)+1\bmod{3}$, where $u^-$ is the unique vertex of $\Gamma^-(u)$.

Next, we define a vertex coloring of $G$. At first, we color the vertices of the directed cycle. If $3\mid \ell$ or $3\mid (\ell +2)$ then $c(v_i)\equiv i$ for $i\in[\ell]$. If $3\mid \ell+1$ then set $c(v_i)\equiv i$ for $i\in[\ell-1]$ and $c(v_{\ell})=2$. After coloring all roots, the remaining vertices are colored based on the coloring function explained in the preceding paragraph. Finally, we have

\begin{itemize}
\item If $3\mid \ell$, then each of the color classes is a \abbr{Qk} of $G$; since they are disjoint, the digraph $G$ contains a \abbr{Qk} of size at most $\frac{n}{3}$.
\item If $3\mid\ell-1$, then the color classes $1$ and $2$ are quasi-kernels of $G$. Even more so, the color class $3$ with the vertex $v_1$ is also a \abbr{Qk}. Thus, by the pigeonhole principle, the digraph $G$ contains a \abbr{Qk} of size at most $\frac{n+1}{3}$.
\item If $3\mid\ell+1$, then color class $2$ is a \abbr{Qk} of $G$. Moreover, the color class $1$ with the vertex $v_{\ell-1}$ is also a \abbr{Qk}, and the color class $3$ with the vertex $v_1$ is also a \abbr{Qk}. Thus by the pigeonhole principle, the digraph $G$ contains a \abbr{Qk} of size at most $\frac{n+2}{3}$.
\end{itemize}
It is easy to see that these bounds are sharp by considering directed cycles.
\end{proof}

\section{A very recent development}
Long after submitting this paper but well before its revision a very interesting new preprint was published on the topic (see \cite{spiro}). We cannot go to details, but we list here the main points. First it proves that every source-free digraph contains a \abbr{QK} of size at most $n-\sqrt{n}.$ While this is even not in the range of  $(1-\epsilon)n$ this is the first genue upper bound which is separated from $n.$ Spiro also introduced the notion of $q$-kernels: this are independent subsets, from where every other vertex can be reached within $q$ steps. (Therefore the \abbr{QK}s are $2$-kernels.) The paper proves that for each $q\ge 3$, in any source-free digraph there always exists a $q$-kernel of size at most $n/2.$ Spiro also introduced the notion of \emph{large \abbr{QK}}: a \abbr{QK} is \emph{large}, if $|Q \cup \Gamma^+(Q)| \ge n/2.$ Then the \emph{large \abbr{QK} conjecture} states that every source-free digraph contains a large \abbr{QK}. And the author proved very interesting proposition, that the small \abbr{QK} conjecture implies the large \abbr{QK} conjecture.

\section{Quasi-kernel in Infinite Graphs}
Many combinatorial objects do not have sensible analogs in the infinite environment. However, for \abbr{Qk}, such an analogue can be found. One can recognize easily, that the Chv\'atal-Lov\'asz Theorem, and even (its special case for tournaments) the Landau theorem (\cite{landau}) do not hold for infinite digraphs (see \cite{quasi}). Indeed, consider the tournament ${\mathbb{Z}}$ on the underlying set $\mathbb{Z}$ with all edges directed upward (${xy} \in E(\mathbb{Z}) \leftrightarrow x < y$). It is clear that there is neither a kernel nor a quasi-kernel in $\mathbb{Z}$.

However, it is easy to see that $\mathbb{Z} =  \Gamma^-(0) \cup \Ga(1)$, and this holds  more generally, as we see next.
\begin{theorem}[Erd\H{o}s and Soukup,~\cite{quasi}]\label{th:in-out}
    In any countable infinite tournament $T$ there exists a \abbr{Qk} or there exist vertices $x$ and $y$ such that $V(T) =  \Gamma^-(x) \cup \Ga(y)$.
\end{theorem}
\begin{proof}
Let $x\in V$ be arbitrary. If  $y\notin \Ga_2{x}$ then $V=\Gamma^-{x} \cup \Gamma^+{y}$.  Indeed,  if $z\notin \Ga{y}$ then $(z,y)\in E(G)$ but $xzy$ is not a directed path of length two in $G$ by the choice of $y$, so $(x,z)\notin E$. Thus $(z,x)\in E$, i.e $z\in \Gamma^-{x} $. Since $z$ was arbitrary, we proved the statement.
\end{proof}

\noindent A similar statement holds more generally in directed graphs.
\begin{theorem}[Erd\H{o}s, Hajnal and Soukup,  {\cite[Theorem 2.1]{quasi}}] \label{th:hajnal}
In any countable infinite directed graph $G$  there exist two disjoint, independent subsets $A$ and $B$ such that $V(G) =  \Gamma^-_2(A) \cup \Ga_2(B)$.
\end{theorem}
\noindent We do not consider this result as a genuine generalization of Theorem~\ref{th:CL} since it can happen that $\Gamma^-(A) \cap \Ga(B) \ne \emptyset$. The paper suggested a conjecture which would be a real generalization

\medskip
\noindent For a digraph $G$, an independent vertex set $K \subseteq V(G)$ is a \emph{sink}, if $\Gamma^-_1(K)=V(G)$. An independent vertex set $Q$ is a \emph{quasi-sink} if $\Gamma^-_2(Q)=V(G)$.
Using these notions we have the following conjecture for countably infinite digraphs:
\begin{conjecture}[\cite{quasi, Fete}]\label{th:inf}
    In any countable infinite directed graph $G$ there exists a vertex partition $V(G)= A \cup B$ such that in $G[A]$ there exists a quasi-sink and in $G[B]$ there exists a quasi-kernel.
\end{conjecture}
\noindent Please observe, that the cut edges between the partition classes are not used.

There are several special classes of countable infinite digraphs where the conjecture has been proved. We give here some examples. Each of the following classes contain infinite digraphs which somewhat ``resemble'' the finite digraphs:
\begin{enumerate}[{\rm (1)}]
\item A digraph whose in-degrees are finite has a \abbr{Qk}.
\item A digraph with a finite chromatic number has a \abbr{Qk}.
\item If for the infinite digraph $G$ there exists an $n\ge 2$ such that the complement of $G$ does not contains $K_n$, then $G$ satisfies Conjecture \ref{th:inf}.
\item Each infinite digraph $G$ whose complement has finite chromatic number, $G$ satisfies Conjecture \ref{th:inf}.
\item Each infinite digraph $G$ whose complement is locally finite (that is, each vertex has finite degree), $G$ satisfies Conjecture \ref{th:inf}.
\end{enumerate}

\bibliographystyle{plain}

\end{document}